\documentclass[oneside, 10pt]{amsart}

\usepackage{amsmath, amsfonts, amssymb, amsthm, graphics, times, lmodern}

\newtheorem{theorem}{Theorem}[section]
\newtheorem{lemma}[theorem]{Lemma}

\theoremstyle{definition}

\numberwithin{equation}{section}

\begin{document}

\title[ Sharp Riesz-Fej\'er inequality for harmonic functions]{Sharp Riesz-Fej\'er inequality for harmonic Hardy spaces }

\author{Petar  Melentijevi\'{c}}
\address{
Faculty of    Mathematics\\
University of Belgrade\\
Studentski trg 16\\
11000 Beograd\\
Serbia\\
}\email{petarmel@matf.bg.ac.rs}

\author{Vladimir Bo\v{z}in}
\address{
	Faculty of    Mathematics\\
	University of Belgrade\\
	Studentski trg 16\\
	11000 Beograd\\
	Serbia\\
}\email{bozinv@turing.mi.sanu.ac.rs}
\thanks{The first author is partially supported by MPNTR grant 174017, Serbia, the second author is partially supported bz MPNTR grant 174032}

\begin{abstract}
We prove sharp version of Riesz-Fej\'er inequality for functions in harmonic Hardy space $h^p(\mathbb{D})$ on the unit disk $\mathbb{D}$, for $p>1,$ thus extending the result from \cite{KPK}  and resolving the posed conjecture.

\end{abstract}



\maketitle

\section{Introduction}

Let $\mathbb{D}$ denote the unit disk in the complex plane. For holomorphic or harmonic function $f$ with $M_{p}(r,f)$ we denote the integral means:

$$M_{p}(r,f)=\bigg(\frac{1}{2\pi}\int_{0}^{2\pi}|f(re^{i \theta})|^pd\theta\bigg)^{\frac{1}{p}}, \quad 0<p<\infty$$

The space of all holomorphic functions for which $M_{p}(r,f)$ is bounded for $0<r<1$ is the Hardy space $H^p(\mathbb{D}),$ while the analogous space of harmonic functions is the harmonic Hardy space $h^p(\mathbb{D}).$ Theory of Hardy spaces is a very well developed; for further background about these spaces, we refer reader, for instance, to the books \cite{KOOSIS} and \cite{PAVLOVIC}.

One of the interesting results in this theory is the following inequality of Riesz and Fej\'er from \cite{RIESZFEJER}:

$$\int_{-1}^{1} |f(r)|^p dr \leq \frac{1}{2} \int_{0}^{2\pi} |f(e^{i\theta})|^p d\theta, $$
that holds for a function $f \in H^{p}(\mathbb{D})$ for every $0<p<\infty,$ where the values $f(e^{i\theta})$ denote the radial limits of the function $f.$

This inequality was generalized in several directions. Let us mention Beckenbach's results: the same inequality holds where in place of $|f|^p$ we have a positive logarithmically subharmonic function. Some of generalizations can be found in \cite{BECKENBACH}, \cite{CALDERON} and \cite{HUBER}. 

A recent significant result is an analog of this inequality for harmonic Hardy spaces, proved by Kayumov et al. Namely, they proved the next version of Riesz-Fej\'er inequality:

$$\int_{-1}^{1}|f(re^{is})|^p dr\leq K_{p}  \int_{0}^{2\pi}|f(e^{i\theta})|^p d\theta,$$
for all $s \in [0,2\pi]$ with $K_{p}=\frac{1}{2\cos^p\frac{\pi}{2p}}$ for $1<p<2$ and $K_{p}=1$ for $p\geq2.$ The inequality is sharp for $p \in (1,2]$ and the authors made a conjecture that the inequality holds with  $K_{p}=\frac{1}{2\cos^p\frac{\pi}{2p}}$ for all $1<p<\infty.$ They also proved $K_{p}\geq \frac{1}{2\cos^p\frac{\pi}{2p}}$ for these $p,$ so the inequality with this $K_{p}$ would be the optimal one. The inequality for $1<p<2$ depends on an inequality of Kalaj, proved in \cite{KALAJ} and Lozinski's inequality from \cite{LOZINSKI}. The proof of the first of these inequalities uses the plurisubharmonic method invented in \cite{HOLLENBECKVERBITSKY}; recent update on this method can be found in \cite{MELENTIJEVICMARKOVIC}. 
The proof of Riesz-Fej\'er inequality for $p>2$ uses a result of Frazer from \cite{FRAZER}.

The purpose of this paper is to prove the sharp version of  Riesz-Fej\'er inequality for harmonic Hardy spaces for every $1<p<\infty$ using Schur test for Poisson extension operator. Namely, we get the following theorem:

\begin{theorem}

	For all $1<p<\infty$ and $f \in h^p(\mathbb{D})$, we have:
$$\int_{-1}^{1}|f(re^{is})|^p dr\leq  \frac{1}{2\cos^p\frac{\pi}{2p}} \int_{0}^{2\pi}|f(e^{i\theta})|^p d\theta,$$
with $s \in [0,2\pi].$
\end{theorem}

Because of the rotational invariance of norm of functions in $h^p(\mathbb{D}),$ we can consider only the case of $s=0,$ without any loss of generality. 

\section{Proof of the main theorem}

We will prove Theorem 1.1 using the following version of Schur test as can be found in \cite{HOWARDSCHEP}:

\begin{lemma}
 Let $X$ and $Y$ be measure spaces equipped with nonnegative, $\sigma-$finite measures and let $T$ be an operator from $L^p(Y)$ to $L^p(X)$ that can be expressed as
 
 $$Tf(x)=\int_{Y} K(x,y)f(y) dy $$ 
 for some nonnegative function $K(x,y).$ The adjoint operator $T^{*}$ is now given by 
 
  $$T^{*}f(y)=\int_{X} K(x,y)f(x) dx. $$ 
 
If we can find a measurable $h$ finite almost everywhere, such that:
$$ T^{*}((Th)^{p-1}) \leq C_{p} h^{p-1},\quad  \text{a.e. on Y} $$

then for all $f \in L^p(Y),$ we have:

$$ \int_{X}|T(f)|^p dx \leq C_{p}\int_{Y}|f|^p dy.  $$
\end{lemma}

We apply the Schur test in the following setting. For spaces $X$ and $Y$ we set $X=[-1,1]$ with Lebesgue measure and $Y=\mathbb{T}=\partial\mathbb{D}$ with normalised arclength measure. Starting from a harmonic $f \in h^p(\mathbb{D}),$ we first get the appropriate $f^*(e^{i\theta}) \in L^p(\mathbb{T}),$ defined by its radial limits. Now, by acting with the operator $T$ of Poisson harmonic extension, we get:

$$  Tf^{*}(r)=\int_{0}^{2\pi}\frac{1-r^2}{1-2r\cos\theta+r^2}f^{*}(e^{i\theta}) \frac{d\theta}{2\pi}, $$
which is equal to $f(r),$ because of harmonicity of $f.$ Hence, we easily see that the optimal constant in Riesz-Fej\'er inequality is equal to the $p-$th power of the operator norm of such $T.$ Since we consider $\mathbb{T}$ with normalised measure we have to find an $h$ such that the constant $C_{p}$ is equal to $\frac{\pi}{\cos^p\frac{\pi}{2p}}.$ Also,  $T$ has  positive kernel $K(r,\theta)=\frac{1-r^2}{1-2r\cos\theta+r^2},$ and therefore, it follows that

$$T^{*}f(e^{i\theta})=\int_{-1}^{1} \frac{1-r^2}{1-2r\cos\theta+r^2}f(r) dr.$$

We will work with $h(z)= \Re (1-z^2)^{-\frac{1}{p}}.$ It is easy to find to its values on the unit circle so that

$$\Re (1-e^{2i\theta})^{-\frac{1}{p}}= \Re (2 \sin \theta e^{i(\theta-\frac{\pi}{2})})^{-\frac{1}{p}}=2^{-\frac{1}{p}}\sin^{-\frac{1}{p}}\theta \cos\bigg(\frac{\pi}{2p}-\frac{\theta}{p}\bigg),$$
for $0\leq\theta\leq \pi,$

$$\Re (1-e^{2i\theta})^{-\frac{1}{p}}= \Re (2 \sin \theta e^{i(\theta-\frac{\pi}{2})})^{-\frac{1}{p}}=2^{-\frac{1}{p}}|\sin\theta|^{-\frac{1}{p}} \cos\bigg(\frac{\pi}{2p}-\frac{\theta-\pi}{p}\bigg),$$
for $\pi \leq\theta\leq 2\pi,$ while on the real line we have

$$ \Re (1-z^2)^{-\frac{1}{p}}= (1-r^2)^{-\frac{1}{p}}.$$

From the fact that

$$ Th(r)=\frac{1}{2\pi}\int_{0}^{2\pi}\frac{1-r^2}{1-2r\cos\theta+r^2}\Re(1-e^{2 i \theta})^{-\frac{1}{p}} d\theta=(1-r^2)^{-\frac{1}{p}},$$
for $-1<r<1,$ we find that

$$ T^{*}((Th)^{p-1})=\int_{-1}^{1} \frac{1-r^2}{1-2r\cos\theta+r^2} (1-r^2)^{-\frac{p-1}{p}} dr     $$
$$=\int_{-1}^{1} \frac{(1-r^2)^{\frac{1}{p}}}{1-2r\cos\theta+r^2}  dr.    $$

 Since $\int_{-1}^{1} \frac{(1-r^2)^{\frac{1}{p}}}{1-2r\cos(\theta+\pi)+r^2} dr=\int_{-1}^{1} \frac{(1-r^2)^{\frac{1}{p}}}{1+2r\cos\theta+r^2} dr, $
and substituting $r$ with $-r$ in the last integral, we easily see that it is enough to prove
 
  $$\int_{-1}^{1} \frac{(1-r^2)^{\frac{1}{p}}}{1-2r\cos\theta+r^2} dr \leq \frac{\pi}{\cos^p{\frac{\pi}{2p}}} 2^{-\frac{p-1}{p}}\sin^{-\frac{p-1}{p}}\theta \cos^{p-1}\bigg(\frac{\pi}{2p}-\frac{\theta}{p}\bigg),$$
 for $0<\theta<\pi, $ i.e.  $T^{*}((Th)^{p-1}) \leq C_{p} h^{p-1} $ almost everywhere on $0<\theta<\pi $ and consequently on the whole domain. 

Introducing a change of variables $\frac{1+r}{1-r}=y\cot\frac{\theta}{2}$ in the integral, we have

$$\int_{-1}^{1} \frac{(1-r^2)^{\frac{1}{p}}}{1-2r\cos\theta+r^2} dr$$
$$=\int_{0}^{+\infty} \frac{\big[1-\big(\frac{y\cot\frac{\theta}{2}-1}{y\cot\frac{\theta}{2}+1}\big)^2\big]^{\frac{1}{p}}}{1-2\cos\theta\frac{y\cot\frac{\theta}{2}-1}{y\cot\frac{\theta}{2}+1}+\big(\frac{y\cot\frac{\theta}{2}-1}{y\cot\frac{\theta}{2}+1}\big)^2} \frac{2\cot\frac{\theta}{2}}{(y\cot\frac{\theta}{2}+1)^2} d y$$
$$=4^{\frac{1}{p}}\cot^{1+\frac{1}{p}}\frac{\theta}{2}\int_{0}^{+\infty} \frac{y^{\frac{1}{p}}}{y^2(1-\cos\theta)\cot^2\frac{\theta}{2}+1+\cos\theta} \frac{dy}{(y\cot\frac{\theta}{2}+1)^{\frac{2}{p}}}$$
$$=\frac{4^{\frac{1}{p}}\cot^{1+\frac{1}{p}}\frac{\theta}{2}}{2\cos^2\frac{\theta}{2}}  \int_{0}^{+\infty} \frac{y^{\frac{1}{p}}}{y^2+1} \frac{dy}{(y\cot\frac{\theta}{2}+1)^{\frac{2}{p}}} .                       $$

Hence, we obtain

$$ \sin^{\frac{p-1}{p}}\theta \int_{-1}^{1} \frac{(1-r^2)^{\frac{1}{p}}}{1-2r\cos\theta+r^2} dr$$
$$=  \frac{4^{\frac{1}{p}}\cot^{1+\frac{1}{p}}\frac{\theta}{2}\sin^{\frac{p-1}{p}}\theta}{2\cos^2\frac{\theta}{2}}  \int_{0}^{+\infty} \frac{y^{\frac{1}{p}}(y^2+1)^{-1}}{(y\cot\frac{\theta}{2}+1)^{\frac{2}{p}}}  dy $$
$$=2^{\frac{1}{p}}\int_{0}^{+\infty} \frac{y^{\frac{1}{p}}(y^2+1)^{-1}}{(y\cos\frac{\theta}{2}+\sin\frac{\theta}{2})^{\frac{2}{p}}}  dy                          $$

The main inequality is equivalent, now, with:

$$\int_{0}^{+\infty} \frac{y^{\frac{1}{p}}(y^2+1)^{-1}}{(y\cos\frac{\theta}{2}+\sin\frac{\theta}{2})^{\frac{2}{p}}}  dy  \leq \frac{\pi}{2\cos^p\frac{\pi}{2p}} \cos^{p-1}\bigg(\frac{\pi}{2p}-\frac{\theta}{p}\bigg)                       $$
 
 Since $\cos^{p-1}\bigg(\frac{\pi}{2p}-\frac{\theta}{p}\bigg) \geq \cos^{p-1}\frac{\pi}{2p},$ it is enough to prove:
 
$$F(\theta)=\int_{0}^{+\infty} \frac{y^{\frac{1}{p}}(y^2+1)^{-1}}{(y\cos\frac{\theta}{2}+\sin\frac{\theta}{2})^{\frac{2}{p}}}  dy  \leq \frac{\pi}{2\cos\frac{\pi}{2p}}. $$

The proof follows from the next two lemmas.

\begin{lemma}
	\label{Lemma 1}

Function $F(\theta)$ is convex on $[0,\pi]$ and

	$$F(0)=F(\pi)=\frac{\pi}{2\cos\frac{\pi}{2p}}.$$	
\end{lemma}
\begin{proof}
First, we rewrite $F(\theta)$ in more suitable form. Changing variable with $x=\arctan y,$ we get:
\begin{multline}
F(\theta)=\int_{0}^{+\infty} \frac{y^{\frac{1}{p}}(y^2+1)^{-1}}{(y\cos\frac{\theta}{2}+\sin\frac{\theta}{2})^{\frac{2}{p}}}  dy =\int_{0}^{\frac{\pi}{2}}  \frac{\tan ^{\frac{1}{p}}x}{(\tan x \cos\frac{\theta}{2}+\sin\frac{\theta}{2})^{\frac{2}{p}}} dx          
\\ =\int_{0}^{\frac{\pi}{2}}  \frac{\tan ^{\frac{1}{p}}x\cos^{\frac{2}{p}}x}{(\sin x \cos\frac{\theta}{2}+\cos x\sin\frac{\theta}{2})^{\frac{2}{p}}} dx   =   \int_{0}^{\frac{\pi}{2}}  \frac{\sin ^{\frac{1}{p}}x\cos^{\frac{1}{p}}x}{\big(\sin (x+\frac{\theta}{2})\big)^{\frac{2}{p}}} dx.                                        
\end{multline}
Differentiating twice with respect to $\theta,$ we get:

$$F''(\theta)=\frac{1}{2p}\int_{0}^{\frac{\pi}{2}}\Phi(x,\theta) dx,$$
which is positive, since the integrand
 $$\Phi(x,\theta)=\frac{\sin ^{\frac{1}{p}}x\cos^{\frac{1}{p}}x}{\big(\sin (x+\frac{\theta}{2})\big)^{2+\frac{2}{p}}} \bigg[\big(1+\frac{2}{p}\big)\cos^2\big(x+\frac{\theta}{2}\big)+\sin^2\big(x+\frac{\theta}{2}\big)\bigg]$$
 is positive for all $x \in [0,\frac{\pi}{2}]$ and $\theta \in [0,\pi].$ Thus, $F(\theta)$ is convex on $[0,\pi].$

By (2.1) and change of variable $x=\frac{\pi}{2}-t,$ we get:
 $$F(0)=\int_{0}^{\frac{\pi}{2}}\sin ^{-\frac{1}{p}}x\cos^{\frac{1}{p}}x dx=\int_{0}^{\frac{\pi}{2}}\sin ^{\frac{1}{p}}t\cos^{-\frac{1}{p}}t dt=F(\pi).$$
 
Also, from the formula for Beta function we have: $F(0)=\frac{1}{2}B(\frac{1}{2}-\frac{1}{2p},\frac{1}{2}+\frac{1}{2p})=\frac{1}{2\sin(\frac{\pi}{2}-\frac{\pi}{2p})}=\frac{1}{2\cos\frac{\pi}{2p}}.$ 

\end{proof}

Using Lemma \ref{Lemma 1}, we easily finish the proof of the main inequality. Since $F(\theta)$  is convex, it attains its maximum at the end of the interval $[0,\pi],$ and by the same lemma its values at $0$ and $\pi$ are both equal to $\frac{1}{2\cos\frac{\pi}{2p}},$
hence $F(\theta) \leq \frac{1}{2\cos\frac{\pi}{2p}}.$

\textbf{Acknowledgements}. We wish to express our gratitude to the anonymous referee for his/her helpful comments that have improved the quality of the paper.

\end{document}